\documentclass{amsart}
\usepackage{latexsym}
\usepackage{amsfonts,amsmath,amssymb,stmaryrd,amscd,amsxtra,amsthm,verbatim,calc}
\usepackage[all]{xy}
\RequirePackage[dvipsnames,usenames]{xcolor}

\usepackage{xypic}
\usepackage{eucal}

\usepackage{amsmath}
\usepackage{amstext}
\usepackage{amssymb}
\usepackage{amsthm}
\usepackage{amscd}

\usepackage{lscape}
\usepackage{longtable}

\usepackage{marginnote}

\usepackage{epsfig}
\let\et=\etexdraw
\def\etexdraw{\drawbb\et}

\theoremstyle{plain}
\newtheorem{thm}{Theorem}[section]
\newtheorem{theorem}{Theorem}[section]
\newtheorem{thm*}{Theorem}
\newtheorem{lem}[thm]{Lemma}

\newtheorem{prop*}[thm*]{Proposition}
\newtheorem{cor}[thm]{Corollary}

\theoremstyle{definition}
\newtheorem{defi}[thm]{Definition}

\newtheorem{ex}[thm]{Example}
\newtheorem{qu}[thm]{Question}

\theoremstyle{remark}

\newtheorem{remark}[thm]{Remark}

\DeclareMathOperator{\Ker}{Ker}
\DeclareMathOperator{\Coker}{Coker}
\DeclareMathOperator{\Image}{Im}

\DeclareMathOperator{\height}{ht}
\DeclareMathOperator{\Hom}{Hom}
\DeclareMathOperator{\Ass}{Ass}
\DeclareMathOperator{\rank}{rank}

\DeclareMathOperator{\Ext}{Ext}

\DeclareMathOperator{\Ann}{ann}
\DeclareMathOperator{\Nil}{Nil}

\DeclareMathOperator{\HH}{H}

\DeclareMathOperator{\Supp}{Supp}

\DeclareMathOperator{\fm}{\mathfrak{m}}

\DeclareMathOperator{\fp}{\mathfrak{p}}

\DeclareMathOperator{\fa}{\mathfrak{a}}

\begin{document}

\title{Multiplicity bounds in prime characteristic}

\dedicatory{Dedicated to Gennady Lyubeznik on the occasion of his sixtieth birthday}

\author{Mordechai Katzman}
\address{Department of Pure Mathematics,
University of Sheffield, Hicks Building, Sheffield S3 7RH, United Kingdom}
\email{M.Katzman@sheffield.ac.uk}

\author{Wenliang Zhang}
\address{Department of Mathematics, Statistics, and Computer Science, University of Illinois at Chicago, 851 S. Morgan Street, Chicago, IL 60607-7045}
\email{wlzhang@uic.edu}

\thanks{
W.Z. is partially supported by NSF grants DMS \#1606414 and DMS \#1752081.}


\begin{abstract}
We extend a result by Huneke and Watanabe (\cite{HunekeWatanabeUpperBoundOfMultiplicity}) bounding the multiplicity of $F$-pure local rings of prime characteristic in terms of their
dimension and embedding dimensions to the case of $F$-injective, generalized Cohen-Macaulay rings. We then produce an upper bound for the multiplicity of any local Cohen-Macaulay ring of prime characteritic
in terms of their dimensions, embedding dimensions and HSL numbers. Finally, we extend the upper bounds for the multiplicity of generalized Cohen-Macaulay rings in characteristic zero which have dense $F$-injective type.
\end{abstract}

\maketitle

\section{Introduction}
In \cite{HunekeWatanabeUpperBoundOfMultiplicity}, Huneke and Watanabe proved that, if $R$ is a noetherian, $F$-pure local ring of dimension $d$ and embedding dimension $v$, 
then $e(R)\leq \binom{v}{d}$ where $e(R)$ denotes the Hilbert-Samuel multiplicity of $R$. The following was left as an open question in \cite[Remark 3.4]{HunekeWatanabeUpperBoundOfMultiplicity}:
\begin{qu}[Huneke-Watanabe]
\label{Huneke-Watanabe question}
Let $R$ be a noetherian $F$-injective local ring with dimension $d$ and embedding dimension $v$. Is it true that $e(R)\leq \binom{v}{d}$?
\end{qu}
In this note, we answer this question in the affirmative when $R$ is generalized Cohen-Macaulay.

\begin{theorem}
\label{thm: bound in GCM case}
Let $R$ be a $d$-dimensional noetherian $F$-injective generalized Cohen-Macaulay local ring of embedding dimension $v$. Then 
\[e(R)\leq \binom{v}{d}.\]
\end{theorem}

Using reduction mod $p$, one can prove an analogous result for generalized Cohen-Macaulay rings of dense $F$-injective type in characteristic 0, {\it cf.} Theorem \ref{Theorem: the bound in characteristic zero}.

We also generalize these result to Cohen-Macaulay, non-$F$-injective rings as follows.
\begin{defi}[cf.~section 4 in \cite{LyubeznikFModulesApplicationsToLocalCohomology}]
Let $A$ be a commutative ring and let $H$ be an $A$-module with Frobenius map $\theta :  H \rightarrow H$ (i.e., an additive map
such that $\theta(a h)=a^p \theta (h)$ for all $a\in A$ and $h\in H$). Write $\Nil H=\{ h\in H \,|\, \theta^e h = 0 \text{ for some } e\geq 0 \}$.
The \emph{Hartshorne-Speiser-Lyubeznik number} (henceforth \emph{abbreviated HSL number}) is defined as
$$\inf \{e\geq 0 \,|\, \theta^e \Nil H = 0\}. $$

The HSL number of a local, Cohen-Macaulay ring $(R, \mathfrak{m})$ is defined as the HSL number of the top local cohomology module $\HH^{\dim R}_\mathfrak{m} (R)$
with its natural Frobenius map.
\end{defi}
For artinian modules over a quotient of a regular ring,  HSL numbers are finite. (\cite[Proposition 4.4]{LyubeznikFModulesApplicationsToLocalCohomology}).

\bigskip
Without the $F$-injectivity assumption, we have the following upper bound in the Cohen-Macaulay case which involves the HSL number of $R$.
\begin{theorem}[Theorem \ref{Theorem: bounds with HSL numbers}]
Assume that $(R,\fm)$ is a reduced, Cohen-Macaulay noetherian local ring of dimension $d$ and embedding dimension $v$.
Let $\eta$ be the HSL number of $R$ and write $Q=p^\eta$.
Then
\[e(R)\leq Q^{v-d}\binom{v}{d}.\]
\end{theorem}
This bound is asymptotically sharp as shown in Remark \ref{remark: asymptotically sharp}.

\section{Bounds on $F$-injective rings}

For each commutative noetherian ring $R$, let $R^o$ denote the set of elements of $R$ that are not contained in any minimal prime ideal of $R$.

\begin{remark}
\label{rem: nonzerodivisor}

If $R$ is a reduced noetherian ring, then each $c\in R^o$ is a non-zero-divisor.
\end{remark}

Given any local ring $(R,\fm)$, we can pass to $S=R[x]_{\fm R[x]}$ which admits an infinite residue field: this does not affect the multiplicity, dimension, embedding dimension and  
Cohen-Macaulyness (cf.~\cite[Lemma 8.4.2]{HunekeSwansonIntegralClosure}). In addition, since $S$ is a faithfully flat extension of $R$,
$\HH^i_{\fm S} (S) = \HH^i_{\fm} (R) \otimes_R S$ and, if $\phi_i : \HH^i_{\fm} (R) \rightarrow \HH^i_{\fm} (R)$ is the natural Frobenius map
induced by the Frobenius map $r\mapsto r^p$ on $R$, then the natural Frobenius map on $\HH^i_{\fm S} (S)$
takes an element $a \otimes x^\alpha$ to $\phi_i(a) \otimes x^{\alpha p}$. Therefore, passing to $S$ preserves HSL numbers (and hence also $F$-injectivity). Therefore, for the purpose of seeking an upper bound of multiplicity, we may assume that that all mentioned local rings $(R,\fm)$ have infinite residue fields; consequently, $\fm$ admits a minimal reduction generated by $\dim R$ elements (cf.~\cite[Proposition 8.3.7]{HunekeSwansonIntegralClosure}).

We begin with a Skoda-type theorem for $F$-injective rings which may be viewed as a generalization of \cite[Theorem 3.2]{HunekeWatanabeUpperBoundOfMultiplicity}.

\begin{thm}
\label{Theorem: power of m in Frobenius closure}
Let $(R,\fm)$ be a commutative noetherian ring of characteristic $p$ and let $\fa$ be an ideal that can be generated by $\ell$ elements. Assume that each $c\in R^o$ is a non-zero-divisor. Then
\[\overline{\fa^{\ell+1}}\subseteq \fa^F,\]
where $\overline{\fa^{\ell+1}}$ is the integral closure of $\fa^{\ell+1}$ and $\fa^F$ the Frobenius closure of $\fa$.
\end{thm}
\begin{proof}
For each $x\in \overline{\fa^{\ell+1}}$ pick $c\in R^o$ such that for $N\gg 1$,
$c x^N \in \fa^{(\ell+1)N}$ (\cite[Corollary 6.8.12]{HunekeSwansonIntegralClosure}). Note that $c$ is a non-zero-divisor by our assumptions. We have
$c x^N \in c(\fa^{(\ell+1)N} : c)\subseteq cR \cap \fa^{(\ell+1)N}$.
An application of the Artin-Rees Lemma gives a $k\geq 1$ such that
$c x^N \in c \fa^{(\ell+1)N-k}$ for all large $N$, and so
$ x^N \in \fa^{(\ell+1)N-k}$ for all large $N$.
For any large enough $N=p^e$ we have
$x^{p^e}\in \fa^{[p^e]}$, i.e., $x$, and hence $\overline{\fa^{d+1}}$  is in the Frobenius closure of $\fa$.
\end{proof}

\begin{cor}
\label{Corollary: multiplicity bound in F-injective rings}
Let $(R,\fm)$ be a $d$-dimensional noetherian local ring of characteristic $p$. Assume that $\fm$ admits a minimal reduction $J$. Then 
\begin{enumerate}

\item[(a)] $\fm^{d+1} \subseteq \overline{\fm^{d+1}} = \overline{J^{d+1}} \subseteq J^{F}$, and

\item[(b)] $e(R)\leq \binom{v}{d} + \ell(J^{F}/J)$.
\end{enumerate}
\end{cor}
\begin{proof}

Since $\fm^{d+1} \subseteq \overline{\fm^{d+1}} = \overline{J^{d+1}}$, (a) follows from Theorem \ref{Theorem: power of m in Frobenius closure}.

For part (b), since $\overline{J^{d+1}}\subseteq J^F$ and $J$ is generated by $d$ elements, we have $\ell(R/J^{F})\leq  \binom{v}{d}$ (as in the proof of
\cite[Theorem 3.1]{HunekeWatanabeUpperBoundOfMultiplicity}). Then
\[e(R)\leq \ell (R/J)= \ell(R/J^{F})+\ell(J^{F}/J)\leq \binom{v}{d} + \ell(J^{F}/J).\]
\end{proof}

\begin{proof}[Proof of Theorem \ref{thm: bound in GCM case}]
Let $\hat{R}$ denote the completion of $R$. Then $R$ is $F$-injective and generalized Cohen-Macaulay if and only if $\hat{R}$ is so, and $e(R)=e(\hat{R})$. Hence we may assume that $R$ is complete. Since $R$ is $F$-injective, it is reduced (\cite[Remark 2.6]{SchwedeZhangBertiniTheorems}) and hence each $c\in R^o$ is a non-zero-divisor by Remark \ref{rem: nonzerodivisor}. It is proved in \cite[Theorem 1.1]{MaBuchsbaum} that a generalized Cohen-Macaulay local ring is $F$-injective if and only if every parameter ideal is Frobenius closed. Let $J$ denote a minimal reduction of $\fm$, then $J^F=J$. 
Our theorem follows immediately from Corollary \ref{Corollary: multiplicity bound in F-injective rings}.
\end{proof}

\section{Bounds on multiplicity using HSL numbers}
\begin{thm}\label{Theorem: bounds with HSL numbers}
Assume that $(R,\fm)$ is a reduced, Cohen-Macaulay noetherian local ring of dimension $d$ and embedding dimension $v$.
Let $\eta$ be the HSL number of $R$ and write $Q=p^\eta$.
Then
$e(R)\leq Q^{v-d}\binom{v}{d}$.
\end{thm}
\begin{proof}
We may assume that $R$ is complete since $e(R)=e(\hat{R})$. Hence $\fm$ admits a minimal reduction $J$ (generated by $d$ elements). We have
$e(R)= \ell (R/J)$, and Theorem \ref{Theorem: power of m in Frobenius closure}
shows that $\fm^{d+1} \subseteq J^F$.
Now $\left( J^F \right)^{[Q]}= J^{[Q]}$ for $Q=p^\eta$ hence
$\left(\fm^{d+1}\right)^{[Q]} \subseteq J^{[Q]}$.

Extend a set of minimal generators $x_1, \dots, x_d$ of $J$ to a minimal set of generators
$x_1, \dots, x_d, y_1, \dots y_{v-d}$ of $\fm$.  Now $R/J^{[Q]}$ is spanned by
monomials
$$x_1^{\gamma_1} \dots x_d^{\gamma_d} y_1^{\alpha_1 Q + \beta_1} \dots y_{v-d}^{\alpha_{v-d} Q + \beta_{v-d}}$$
where $0\leq \gamma_1, \dots \gamma_d, \beta_1, \dots, \beta_{v-d} < Q$ and
$0\leq \alpha_1+\dots+\alpha_{v-d}<d+1$. The number of such monomials is
$Q^v \binom{v}{d}$ and so $\ell(R/J^{[Q]})\leq Q^v \binom{v}{d}$.

Note that as $J$ is generated by a regular sequence,
$ \ell(R/J^{[Q]}) = Q^d  \ell(R/J)$ and
we conclude that
$$\ell(R/J) = \ell(R/J^{[Q]}) /Q^d \leq  Q^{v-d} \binom{v}{d} .$$
\end{proof}

\begin{remark}
\label{remark: asymptotically sharp}
  The next family of examples shows that the  bound in  Theorem \ref{Theorem: bounds with HSL numbers} is asymptotically sharp.
  
  Let $\mathbb{F}$ be a field of prime characteristic $p$, let $n\geq 2$, and
  let $S$ be $\mathbb{F}[x_1, \dots, x_n]$.  Let $\mathfrak{m}=(x_1, \dots, x_n)S$, and let $E$ denote the injective hull of the residue field of $S_{\mathfrak{m}}$.

  Define $f=\sum_{i=1}^n x_1^p \dots x_{i-1}^p x_i x_{i+1}^p \dots x_n^p$ and $h=x_1 \dots x_{n-1}$. We claim that $f$ is square-free: if this is not the case
  write $f=r^\alpha s$ where $r$ is irreducible of positive degree, and $\alpha\geq 2$.
  Let $\partial$ denote the partial derivative with respect to $x_n$.
  Note that $\partial f= h^p$ and so
  $$ h^p = \alpha r^{\alpha-1} (\partial r) s + r^\alpha  (\partial s) = r^{\alpha-1}\left(  \alpha (\partial r) s + r  (\partial s) \right).$$
  We deduce that $r$ divides $h$,  but this would imply that $x_i^2$ divides all terms of $f$ for some $1\leq i\leq n-1$, which is false.
  We conclude that $S/fS$ is reduced.

Let $R$ be the localization of $S/fS$ at $\mathfrak{m}$. We compute next the HSL number $\eta$ of $R$ using the
method described in sections 4 and 5 in \cite{KatzmanParameterTestIdealOfCMRings}.
  It is not hard to show that $\HH^{n-1}_{\mathfrak{m}}(R) \cong \Ann_E f$ where $E=\HH^n_{\mathfrak{m}}(S)$,
  and that, after identifying these, the natural Frobenius action on
  $\Ann_E f$ is given by $f^{p-1} T$ where $T$ is the natural Frobenius action on $E$.

  To find the HSL number $\eta$ of  $\HH^{n-1}_{\mathfrak{m}}(R)$ 
  we readily compute
  $I_1(f)$ (cf.~\cite[Proposition 5.4]{KatzmanParameterTestIdealOfCMRings}) to be the ideal generated by
  $\{ x_1 \dots x_{i-1} x_{i+1} \dots x_n \,|\, 1\leq i \leq n\}$ and
  \begin{eqnarray*}
    I_2(f^{p+1})& = &I_1\left( f I_1(f) \right)\\
                &=& \sum_{i=1}^n I_1\big(
                    \sum_{j=1}^{i-1} x_1^{p+1} \dots x_{j-1}^{p+1} x_j^2 x_{j+1}^{p+1} \dots x_{i-1}^{p+1} x_{i}^{p} x_{i+1}^{p+1} \dots x_n^{p+1}\\
                    &+& x_1^{p+1} \dots x_{i-1}^{p+1} x_{i} x_{i+1}^{p+1} \dots x_n^{p+1}\\
                    &+& \sum_{j=i+1}^{n} x_1^{p+1} \dots x_{i-1}^{p+1} x_i^p x_{i+1}^{p+1} \dots x_{j-1}^{p+1} x_{j}^{2} x_{j+1}^{p+1} \dots x_n^{p+1} \big)\\
    &=& I_1(f) \\
  \end{eqnarray*}
  and we deduce that $\eta=1$.

  We now compute
  $$ \Gamma_{n,p}:=\frac{\deg f}{ \binom{n}{n-1} p^\eta}=\frac{(n-1)p+1}{np} .$$
  We have $\lim_{n \rightarrow \infty} \Gamma_{n,p}=1$ and $\lim_{p \rightarrow \infty} \Gamma_{n,p}=(n-1)/n$,
  so we can find  values of $\Gamma_{n,p}$ arbitrarily close to 1. 

\end{remark}

\section{Examples}
The injectivity of the natural Frobenius action on the top local cohomology $H^d_{\fm}(R)$ does {\it not} imply $e(R)\leq \binom{v}{d}$ as shown by the following example.

\begin{ex}
   Let $S=\mathbb{Z}/2\mathbb{Z}[x,y,u,v]$, let $\mathfrak{m}$ be its ideal generated by the variables,
define
$I=(v, x)\cap  (u, x)\cap (v, y)\cap (u, y) \cap  (y, x)\cap   (v, u) \cap (y - u, x - v) =\left( xv(y-u), yu(x-v), yuv(y-u), xuv(x-v) \right)$,
and let $R=S/I$: this is a reduced
2-dimensional ring.

We compute  the following graded $S$-free resolution of $I$
$$
  \xymatrix{
    0 \ar[r] & S(-6) \ar[r]^{B} & S^4(-5) \ar[r]^{A} & S^2(-3) \oplus S^2(-4)  \ar[r] & I \ar[r] & 0\\
  }
$$
where
$$
  A=\left[\begin{array}{cccc}
            u(x-v) & yu & 0 &0\\
            0&0&xv&v(y-u)\\
            0&-x&0&v-x\\
            u-y&0&-y&0
    \end{array}\right], \quad
  B=\left[\begin{array}{c}
y\\v-x\\u-y\\x\\
    \end{array}\right]
$$
and note that $R$ has projective dimension 3, hence depth 1 and so it is not Cohen-Macaulay.
Also, we can read the Hilbert series of $R$ from its graded resolution and we obtain
$$\frac{1-2t^3-2t^4+4t^5-t^6}{(1-t)^4}=\frac{1+2t+3t^2+2t^3-t^4}{(1-t^2)}$$
and so the multiplicity of $R$ is $1+2+3+2-1=7$ exceeding $\binom{4}{2}=6$ (cf.~\cite[\S 6.1.1]{HerzogHibiMonomialIdeals}.)

Note that $R$ is not $F$-injective, but the natural Frobenius action on the top local cohomology module is injective.
\end{ex}

  From the proof of Theorem \ref{thm: bound in GCM case}, we can see that if a minimal reduction of the maximal ideal in an $F$-injective local ring $R$ is Frobenius closed then the bound $e(R)\leq \binom{v}{d}$ will hold. Hence we may ask whether minimal reductions would be Frobenius-closed in such rings (cf.~Theorem 6.5 and Problem 3 in \cite{QuyShimomotoFinjectivity}).
 However,  the following example shows this not to be the case.

\begin{ex}
  Let $S=\mathbb{Z}/2\mathbb{Z}[x,y,u,v,w]$, let $\mathfrak{m}$ be its ideal generated by the variables and let
  $I_1=(x,y) \cap  (x+y, u+w, v+w)$, $I_2=(u,v,w) \cap (x,u,v) \cap (y,u,v)=(u,v,xyw)$, and $I=I_1\cap I_2$.
  Fedder's Criterion \cite[Proposition 1.7]{FedderFPureRat} shows that $S/I_1$, $S/I_2$ and $S/(I_1+I_2)$ are $F$-pure,
  and  \cite[Theorem 5.6]{QuyShimomotoFinjectivity} implies that $S/I$ is F-injective. Also, $S/I$ is almost Cohen-Macaulay: it is 3-dimensional and its localization at $\mathfrak{m}$ has depth 2.

  Its not hard to check that the ideal $J$ generated by the images in $S/I$ of $w, y+v, x+u$ is a minimal reduction.
  However $J^F\neq J$: while $v^2\notin J$, we have
  $$v^4= xyw^2 +v^2(y+v)^2  +yvw(x+y) +(v+w) (y^2v+xyw), $$
  hence $v^2\in J^F\setminus J$.

\end{ex}

\section{Bounds in Characteristic zero}

Throughout this section $K$ will denote a field of characteristic zero,
$T=K[x_1, \dots, x_n]$,
$R$ will denote the finitely generated $K$-algebra
$R=T/I$ for some ideal $I\subseteq T$, and $\mathfrak{m}=(x_1, \dots, x_n)R$; $d$ and $v$ will denote the dimension and embedding dimension, respectively, of $R_{\mathfrak{m}}$.
We also choose $\mathbf{y}=y_1, \dots, y_d\in \mathfrak{m}$  whose images in $R_{\mathfrak{m}}$   form a minimal reduction of $\mathfrak{m}R_\mathfrak{m}$.

We may, and do assume that the only maximal ideal containing $\mathbf{y}$ is $\mathfrak{m}$. Otherwise, if $\fm_1, \dots, \fm_t$ are all the maximal ideals
  distinct from $\fm$ which contain $\mathbf{y}$, we can pick $f\in (\fm_1\cap  \dots \cap \fm_t) \setminus \fm$, and now the only
  maximal ideal containing $\mathbf{y}$ in $R_f$ is $\fm R_f$.
  We may now replace $R$ with $R^\prime=K[x_1, \dots, x_n,x_{n+1}]/I+\langle x_{n+1} f-1 \rangle\cong  R_f$ and
  since $R_{\mathfrak{m}}= (R_f)_{\mathfrak{m}}$ we are not affecting any local issues.

The main tool used in this section descent techniques described in
  \cite{HochsterHunekeTightClosureInEqualCharactersticZero}. We start by introducing a flavour of it useful for our purposes.

\begin{defi}\label{Definition: descent}
By \emph{descent objects} we mean
 \begin{itemize}
 \item[(1)] a finitely generated $K$-algebra $R$ as above,
 \item[(2)] a finite set of finitely generated $T$-modules,
 \item[(3)] a finite set of $T$ linear maps between $T$-modules in (2),
 \item[(4)] a finite set of finite complexes involving maps in (3),
 \end{itemize}

By \emph{descent data} for these descent objects we mean
 \begin{itemize}
 \item[(a)] A finitely generated $\mathbb{Z}$-subalgebra $A$ of $K$,  $T_A=A[x_1, \dots, x_n]$, $I_A\subseteq T_A$ such that with $R_A=T_A/I_A$
    \begin{itemize}
    \item[$\bullet$] $R_A \subseteq R$ induces an isomorphism $R_A\otimes_A K \cong R\otimes_A K=R$, and
    \item[$\bullet$]  $R_A$ is $A$-free.
    \end{itemize}
  \item[(b)] For each $M$ in (2), a finitely generated free $A$-submodule $M_A\subseteq M$ such that this inclusion induces an isomorphism
    $M_A\otimes_A K \cong M\otimes_A K=M$.
  \item[(c)] For every $\phi : M \rightarrow N$ in (3) an $A$ linear map $\phi_A : M_A \rightarrow N_A$ such that
    \begin{itemize}
    \item[$\bullet$]  $\phi_A \otimes 1: M_A \otimes_A K \rightarrow N_A \otimes_A K$ is the map $\phi$, and
    \item[$\bullet$]  $\Image \phi$, $\Ker \phi$ and $\Coker \phi$ are $A$-free.
    \end{itemize}
  \item[(d)] For every homological complex
  $$\mathcal{C}_\bullet = \dots \xrightarrow{\partial_{i+2}} C_{i+1} \xrightarrow{\partial_{i+1}} C_i \xrightarrow{\partial_{i}} \dots $$
  in (4), an homological complex
  $${\mathcal{C}_A}_\bullet= \dots \xrightarrow{(\partial_{i+2})_A} (C_{i+1})_A \xrightarrow{(\partial_{i+1})_A} (C_i)_A \xrightarrow{(\partial_{i})} \dots $$
  such that $\HH_i({\mathcal{C}_A} \otimes_A K)=\HH_i({\mathcal{C}_A} )\otimes_A K$.
  For every cohomological complex in (4), a similar corresponding contruction.
\end{itemize}
\end{defi}

Descent data exist: see \cite[Chapter 2]{HochsterHunekeTightClosureInEqualCharactersticZero}.

Notice that for any maximal ideal $\mathfrak{p}\subset A$, the fiber $\kappa(\mathfrak{p})=A/\mathfrak{p}$
is a finite field. Given any property $\mathcal{P}$ of rings of prime characteristic, we say that $R$ as in the definition above as
\emph{dense $\mathcal{P}$ type} if there exists descent data $(A, R_A)$ and such that for all maximal ideals $\mathfrak{p}\subset A$ the fiber
$R_A \otimes_A \kappa(\mathfrak{p})$ has property $\mathcal{P}$.

Notice also that for any complex $\mathcal{C}$ of free $A$ modules where the kernels and cokernels of all maps are $A$-free (as in Definition \ref{Definition: descent}(c) and (d)),
$\HH_i (\mathcal{C} \otimes_A \kappa(\fp)) = \HH_i (\mathcal{C}) \otimes_A \kappa(\fp)$.

The main result in this section is the following theorem.

\begin{thm}\label{Theorem: the bound in characteristic zero}
If $R_\mathfrak{m}$ is Cohen-Macaulay on the punctured spectrum and has dense $F$-injective type, then
  $e(R_\mathfrak{m})\leq \binom{v}{d}$.
\end{thm}

\begin{lem}\label{Lemma: descent properties}
  There exists descent data $(A, R_A)$ for $R$ with the following properties.
  \begin{enumerate}
  \item[(a)] $y_1, \dots, y_d\in R_A$,
  \item[(b)] for all maximal ideals  $\mathfrak{p}\subset A$ the images of $y_1, \dots, y_d$ in  $R_{\kappa(\mathfrak{p})}$  are a minimal reduction of $\mathfrak{m} R_{\kappa(\mathfrak{p})}$,
  \item[(c)] if $R_\mathfrak{m}$ is Cohen-Macaulay on its punctured spectrum, so is $R_{\kappa(\mathfrak{p})}$ for all maximal ideals  $\mathfrak{p}\subset A$.
  \item[(d)] if $R_\mathfrak{m}$ is unmixed, so is $R_{\fp}$ for all maximal ideals  $\mathfrak{p}\subset A$.
  \end{enumerate}
\end{lem}

\begin{proof}
  Start with some descent data $(A, R_A)$ where $A$ contains all K-coefficients among a set of generators $g_1, \dots, g_\mu$ of $I$, $I_A$ is the ideal of $A[x_1, \dots, x_n]$
  generated by  $g_1, \dots, g_\mu$ and
  $R_A=A[x_1, \dots, x_n]/I_A$. Let $\mathbf{x}$  denote $(x_1, \dots, x_n)$.
  For (a) write $y_i=Q_i(x_1, \dots, x_n)+I$ for all $1\leq i\leq d$ and extend $A$ to include all the K-coefficients in $Q_1, \dots, Q_d$.

  Assume that $\mathfrak{m}^{s+1} \subseteq \mathbf{y}\mathfrak{m}^s$ for some $s$. Write each monomial of degree $s+1$ in the form
  $r_1(\mathbf{x}) Q_1(\mathbf{x}) + \dots + r_d(\mathbf{x}) Q_d(\mathbf{x}) + a(\mathbf{x})$ where $r_1, \dots, r_d$ are polynomials of degrees at least $s$
  and $a(\mathbf{x})\in I$; enlarge $A$ to include all the $K$-coefficients of $r_1, \dots, r_d, a$. With this enlarged $A$ we have
  $(\mathbf{x}R_A)^{s+1} \subseteq (\mathbf{y} R_A) (\mathbf{x} R_A)^s$ and tensoring with any $\kappa(\fp)$ gives
  $( \mathbf{x}  R_{\kappa(\fp)})^{s+1} \subseteq (\mathbf{y} R_{\kappa(\fp)}) (\mathbf{x} R_\kappa(\fp))^s$.

  If $R_\mathfrak{m}$ is Cohen-Macaulay on its punctured spectrum, then we can find a localization of $R$ at one element whose only point at which it can fail
  to be non-Cohen-Macaulay is $\fm$. After adding a new variable to $R$ as at the beginning of this section, we may assume that the non-Cohen-Macaulay locus of $R$ is contained in $\{ \fm \}$.
  The hypothesis in (c) is now equivalent to the existence of a $k\geq 1$ such that $\mathfrak{m}^k \Ext_T^i(R,T)=0$ for all $\height I < i \leq n$.
  Let $\mathcal{F}$ be a free $T$-resolution of $R$. 
  Include $\fm$, $\mathcal{F}$ and $\mathcal{C}=\Hom(\mathcal{F}, T)$ in the descent objects.
  Now, with the corresponding descent data,
  $\mathcal{F}_A$ is a $T_A$-free resolution of $R_A$.
  Localize $A$ at one element, if necessary, so that $\fm_A^k  \Ext_{T_A}^i (R_A, T_A)$ is $A$-free for all $\height I < i \leq n$.
  Fix any  $\height I < i \leq n$; we have
  $$ \Ext_{T_A}^i (R_A, T_A) \otimes_A K=\HH^i( \Hom(\mathcal{F_A}, T_A) ) \otimes_A K= \HH^i( \mathcal{C_A} ) \otimes_A K = \HH^i ( \mathcal{C} ) $$
  and hence $\fm_A^k  \Ext_{T_A}^i (R_A, T_A) \otimes_A K =0$ so $\fm_A^k  \Ext_{T_A}^i (R_A, T_A)=0$. Now for any maximal ideal $\fp \subset A$,
  $\fm_{\kappa(\fp)}^k  \Ext_{T_{\kappa(\fp)}}^i (R_{\kappa(\fp)}, T_{\kappa(\fp)})=0$, and hence
  $R_{\kappa(\fp)}$ is Cohen-Macaulay on its punctured spectrum.

  The last statement is \cite[Theorem 2.3.9]{HochsterHunekeTightClosureInEqualCharactersticZero}.

\end{proof}

\begin{proof}[Proof of Theorem \ref{Theorem: the bound in characteristic zero}]
Using \cite[Theorem 4.6.4]{BrunsHerzog} we write
$e(R_{\fm})= \chi(\mathbf{y}; R_{\fm})$, and using the fact that $R$ was constructed so that $\fm$ is the only maximal ideal
containing $\mathbf{y}$, we deduce that
$e(R_{\fm})= \chi(\mathbf{y}; R)= \sum_{i=0}^d (-1)^i \ell_R \HH_i( \mathbf{y}, R)$.
We add to the descent objects in Lemma \ref{Lemma: descent properties} the Koszul complex $\mathcal{K}_\bullet(\mathbf{y}; R)$ and extend the descent data
in Lemma \ref{Lemma: descent properties} to cater for these.

For all $0\leq i\leq d$ we have $\HH_i(\mathbf{y}; R) \cong \HH_i(\mathbf{y}; R_A) \otimes_A K$ and
$\ell \left( \HH_i(\mathbf{y}; R) \right) = \rank \HH_i(\mathbf{y}; R_A)$.

Pick any maximal ideal $\fp \subset A$.
We have $\HH_i(\mathbf{y}; R_A)\otimes_A \kappa(\fp) \cong \HH_i(\mathbf{y}; R_{\kappa(\fp)})$.

Note that  that $\HH_i(\mathbf{y}; R_{\kappa(\fp)})$ is only supported at $\fm R_{\kappa(\fp)}$. Otherwise, we can find an $x\in \fm R_{\kappa(\fp)}$
such that
$0\neq \HH_i(\mathbf{y}; R_{\kappa(\fp)})_x \cong \HH_i(\mathbf{y}; R_A)_x \otimes_A \kappa(\fp)$ , hence
$\HH_i(\mathbf{y}; R_A)_x \neq 0$ and
$(\HH_i(\mathbf{y}; R_A) \otimes_A K)_x \cong \HH_i(\mathbf{y}; R)_x=0$,
contradicting the fact that $\Supp \HH_i(\mathbf{y}; R) \subseteq \{ \fm \}$.

Now
\begin{align*}
e((R_{\kappa(\fp)})_{\fm})= \chi(\mathbf{y}; (R_{\kappa(\fp)})_{\fm})&= \chi(\mathbf{y}; R_{\kappa(\fp)})\\
&= \sum_{i=0}^d (-1)^i \ell_R \HH_i( \mathbf{y}, R_{\kappa(\fp)})\\
&=\sum_{i=0}^d (-1)^i \rank \HH_i( \mathbf{y}, R_A)
\end{align*}
and so
Theorem \ref{thm: bound in GCM case} implies that $e(R_{\fm})=e((R_{\kappa(\fp)})_{\fm})\leq \binom{v}{d}$.

\end{proof}

\begin{remark}
In \cite{SchwedeFInjectiveAreDuBois} it is conjectured that being a $K$-algebra with dense $F$-injective type is equivalent to being a Du Bois singularity. Recently, the multiplicity of
Cohen-Macaulay Du Bois singularities has been bounded by $\binom{v}{d}$ (see \cite{Shibata2017}) and hence the results of this section provide further evidence for the conjecture above.
\end{remark}

\bibliographystyle{skalpha}
\bibliography{KatzmanBib}

\end{document}